\documentclass[a4paper, 12pt]{amsart}
\usepackage{amsmath}
\usepackage{amssymb}
\usepackage{enumerate}
\usepackage{amsthm}
\theoremstyle{definition}
\newtheorem{theorem}{Theorem}[section]
\newtheorem{definition}[theorem]{Definition}
\newtheorem{lemma}[theorem]{Lemma}
 
\numberwithin{equation}{section}
\newcommand{\myr}{r}
\newcommand{\myb}{b}
\newcommand{\myc}{c}
\newcommand{\mya}{a}
\newcommand{\Y}{Q}
  
\newcommand{\real}{\mathbb R}
\newcommand{\segment}{\phi}
\newcommand{\Segment}{\Phi}
\newcommand{\uds}{U}
\title{Differentiability inside sets with upper Minkowski dimension one.}
\author{Michael Dymond}
\address{School of Mathematics, University of Birmingham,
Birmingham
B15 2TT
UK}
\email{dymondm@maths.bham.ac.uk}
\author{Olga Maleva}
\address{School of Mathematics, University of Birmingham,
Birmingham
B15 2TT
UK}
\email{O.Maleva@bham.ac.uk}
\begin{document}
\begin{abstract} 
We show that every finite-dimensional Euclidean space contains compact universal differentiability sets of upper Minkowski dimension one. In other words, there 
are compact sets $S$ of upper Minkowski dimension one such that every Lipschitz function defined on the whole space is differentiable inside $S$. Such sets are constructed explicitly.
\end{abstract}
\maketitle
\pagestyle{myheadings}
\markboth{Differentiability inside sets with upper Minkowski dimension one}{Michael Dymond and Olga Maleva}
 
\section{Introduction}\label{sec:1}
Lipschitz functions on Banach spaces have somewhat strong differentiability properties. Rademacher's Theorem is a classical result and states that a Lipschitz function on a Euclidean space is differentiable almost everywhere with respect to the Lebesgue measure \cite[p.~100]{mattila99}. A more recent theorem, due to Preiss and published in 1990, asserts that every Lipschitz function on a Banach space $X$, with norm differentiable away from the origin, is differentiable on a dense subset of $X$ \cite{preiss90}.  

For Lipschitz functions on $\mathbb{R}$, a converse statement to Rademacher's Theorem is true; for every subset $N$ of $\mathbb{R}$ with Lebesgue measure zero, there exists a Lipschitz function $f=f_{N}$ that is nowhere differentiable on $N$. This fact is proved in \cite{zahorski46}, where a full characterisation of the possible sets of non-differentiability of a Lipschitz function on $\mathbb{R}$ is given. 

The converse statement to Rademacher's Theorem fails in higher dimensions. In \cite{preiss90}, it is proved that there exists a subset of $\mathbb{R}^{2}$, with Lebesgue measure zero, which contains a point of differentiability of every Lipschitz function on $\mathbb{R}^{2}$. Sets containing a point of differentiability of every Lipschitz function are said to have the universal differentiability property and are called universal differentiability sets. This terminology was introduced by Dor\'{e} and Maleva and first appeared in \cite{doremaleva2}.

As alluded to above, subsets of $\mathbb{R}^{d}$ can be distinguished according to their Lebesgue measure. The Lebesgue null, universal differentiability set given in \cite{preiss90} contains every line segment between pairs of points with rational co-ordinates. In some sense, this set is still rather large; the closure of this set is the whole of $\mathbb{R}^{2}$. Recent work of Dor\'{e} and Maleva has uncovered much smaller sets in $\mathbb{R}^{d}$ with the universal differentiability property. A compact universal differentiability set with Lebesgue measure zero is constructed in \cite{doremaleva1}. 

In order to detect smaller universal differentiability sets, we must appeal to some other notion of size, rather than the Lebesgue measure. In the theory of fractal geometry, the size of a set is often ascertained by 
dimension. The Hausdorff dimension, based on a construction of Carath\'{e}odory is the oldest and perhaps the most important example of a dimension \cite[p. 27]{falconer03}. It is defined for all subsets of a Euclidean space according to the Hausdorff measure; see \cite{mattila99} or \cite{falconer03} for a complete definition. In \cite{doremaleva2}, a compact universal differentiability set with Hausdorff dimension one is constructed. This result is optimal in the sense that any universal differentiability set in a Euclidean space must have Hausdorff dimension at least one \cite[Lemma~1.5]{doremaleva2}.
Finally, it is proved in \cite{doremaleva3} that every non-zero Banach space with separable dual contains a closed and bounded universal differentiability set of Hausdorff dimension one. This is a generalisation of the main result of \cite{doremaleva2} for infinite dimensional Banach spaces.

The Minkowski dimensions of a bounded subset of $\mathbb{R}^{d}$ are closely related to the Hausdorff dimension. Whilst the Hausdorff dimension of a set is based on coverings by sets of arbitrarily small diameter, the Minkowski dimensions are defined similarly according to coverings by sets of the same small diameter. For this reason, the Minkowski dimension is often referred to as the box-counting dimension \cite[p. 41]{falconer03}. The definition below follows \cite[p.~76-77]{mattila99}. 
\begin{definition}\label{def:minkowskidimension}
Given a bounded subset $A$ of $\mathbb{R}^{d}$ and $\epsilon>0$, we define $N_{\epsilon}(A)$ to be the minimal number of balls of radius $\epsilon$ required to cover $A$. That is, $N_{\epsilon}(A)$ is the smallest integer $n$ for which there exists balls $B_{1},\ldots,B_{n}\subseteq\mathbb{R}^{d}$, each of radius $\epsilon$, such that $A\subseteq\cup_{i}B_{i}$.

The lower Minkowski dimension of $A$ is then defined by
\begin{equation}\label{eq:lowerminkowskidimension}
\underline{\dim}_{M}(A)=\inf\left\{s>0:\liminf_{\epsilon\to{0}}N_{\epsilon}(A)\epsilon^{s}=0\right\},
\end{equation}
and the upper Minkowski dimension of $A$ is given by
\begin{equation}\label{eq:upperminkowksidimension}
\overline{\dim}_{M}(A)=\inf\left\{s>0:\limsup_{\epsilon\to{0}}N_{\epsilon}(A)\epsilon^{s}=0\right\}.
\end{equation}
\end{definition}
Writing $\dim_{H}$ for the Hausdorff dimension, it is readily verified that $\dim_{H}(A)\leq\underline{\dim}_{M}(A)\leq\overline{\dim}_{M}(A)$ for all bounded $A\subseteq\mathbb{R}^{d}$. The Hausdorff dimension and Minkowski dimensions can be very different: For example, a countable dense subset of a ball in $\mathbb{R}^{d}$ has Hausdorff dimension $0$ whilst having the maximum upper and lower Minkowski dimension $d$. A construction of a set having lower Minkowski dimension strictly less than its upper Minkowski dimension is given in \cite[p. 77]{mattila99}. It is worth noting that the Minkowski dimension behaves nicely with respect to closures; we have that $\overline{\dim}_{M}(A)=\overline{\dim}_{M}(\mbox{Clos}(A))$ and $\underline{\dim}_{M}(A)=\underline{\dim}_{M}(\mbox{Clos}(A))$. 

In the present paper, we verify the existence of a compact universal differentiability set with upper and lower Minkowski dimension one in $\mathbb{R}^{d}$ for $d\geq{2}$. Such a set is constructed explicitly. This is an improvement on the result of \cite{doremaleva2}, where a compact universal differentiability set of Hausdorff dimension one is given. The construction in \cite{doremaleva2} involves considering a $G_{\delta}$ set $O$ 
of Hausdorff dimension one,
containing all line segments between points belonging to a countable dense subset $R$ of the unit ball in $\mathbb{R}^{d}$ (and hence the Minkowski dimension of $O$ is equal to $d$). The set $O$ can be expressed as $O=\bigcap_{k=1}^{\infty}O_{k}$, where each $O_{k}$ is an open subset of $\mathbb{R}^{d}$ and $O_{k+1}\subseteq{O_{k}}$. For each $k\geq{1}$, a set $R_{k}$ is defined consisting of a finite union of line segments between points from $R$. Since $R_{k}$ is then a closed subset of $O$, it is possible to choose $w_{k}>0$ such that $\overline{B}_{w_{k}}(R_{k})\subseteq{O_{k}}$. The final sets $\uds_{\lambda}$ are then defined by 
\begin{equation*}
\uds_{\lambda}=\bigcap_{k=1}^{\infty}\bigcup_{k\leq{n}\leq{(1+\lambda)k}}\overline{B}_{\lambda{w_{k}}}(R_{k})
\end{equation*}
for $\lambda\in[0,1]$. Observe that, for each $k\geq{1}$ and $\lambda\in[0,1]$ the closed set $\bigcup_{k\leq{n}\leq{(1+\lambda)k}}\overline{B}_{\lambda{w_{k}}}(R_{k})$ is contained in $O_{k}$. Consequently, 
$\dim_H(\uds_\lambda)\le\dim_H(O)=1$.
There is no non-trivial upper bound for the Minkowski dimensions of the sets $\uds_{\lambda}$ constructed in \cite{doremaleva2}.
For constructing a universal differentiability set with upper or lower Minkowski dimension one, the approach of \cite{doremaleva2} fails because the set $O$ has the maximum upper and lower Minkowksi dimension, $d$.

To get a set of lower Minkowski dimension one it would be enough to control the number of
$\delta$-cubes (this will refer to  a cube with side equal $2\delta$) for a specific sequence $\delta_n\searrow 0$. Assume $p>1$ is a fixed number and we want to make sure that the set to be constructed has lower Minkowski dimension less than $p$.
Imagine that we have reached the $n$th step of the construction where we require that $C_n$ is an upper estimate for a number of $\delta_n$-cubes needed to cover
the final set $W$, and $C_n\delta_n^p<1$. The idea for the next step 
is to divide each $\delta_n$-cube by a $K_n\times\dots\times K_n$ grid into smaller $\delta_{n+1}={\delta_n}/{K_n}$-cubes.
If $K_n$ is big enough, then as $\delta_n^p/\delta_{n+1}^p=K_n^p$, we are free to choose inside the given $\delta_n$-cube any number of $\delta_{n+1}$-boxes up
to $K_n^p$. This could, for example, be $K_n(\log K_n)^{M_n}\ll K_n^p$ for any fixed $p>1$, see inequality
\eqref{eq:pf-cond-boxes}; $K_n=\Y^{s_n}$
and $|\mathcal E_n|\le s_n^{2d}$, 
$M_n\ll\frac{s_n}{\log s_n}$. 
We then have that the product of the 
total number of $\delta_{n+1}$-cubes needed to cover $W$ by $\delta_{n+1}^p$ is bounded by $1$ from above as 
well. Since this is satisfied for all $n$, we conclude that $\underline{\dim}_M(W)\le p$. 
Since this is true for every $p>1$, 
we obtain a set of lower Minkowski dimension $1$.

Getting $\overline{\dim}_M(W)\le1$ is less clear. As $n$ grows, the number $K_n$ has to tend to infinity or otherwise we would get many points of porosity inside $W$,
see below for the definition and discussion of porosity. 
In order to prove that $\overline{\dim}_M(W)\le p$ we should be able to show that there exists $\delta_0>0$ such that for every $\delta\in(0,\delta_0)$ the set $W$
can be covered by a controlled number $N_\delta$ of $\delta$-boxes. In other words, $N_\delta\delta^p$ should be bounded for all $\delta$ below certain threshold. Choosing $n$ such that
$\delta_{n+1}<\delta\le\delta_n$ gives 
$N_\delta\delta^p\le N_{\delta_{n+1}}\delta_n^p= N_{\delta_{n+1}}\delta_{n+1}^pK_n^p$,
and the factor $K_n^p\to\infty$ makes it impossible to have a constant upper estimate 
for $N_\delta\delta^p$. The idea here is that we need to leave a ``gap'' for an unbounded sequence in the estimate for $N_{\delta_n}\delta_n^p$ and to make sure that $K_n^p$ fits inside that gap. The realisation of that gap is the inequality
\eqref{eq:claim}.

To finish the introduction, let us briefly explain why we should be concerned about
porosity points. A point $x\in W$ is called its porosity point if there exists 
$\lambda>0$ such that for any  $r>0$ there is a point $y\in B(x,r)$ such that
$B(y,\lambda\|y-x\|)\cap W=\varnothing$. If $x$ is a porosity point of $W$ then 
the distance to $W$,
$f(\cdot)=\mathrm{dist}(\cdot,W)$, is a $1$-Lipschitz function not differentiable at $x$.
Since our aim is to construct a universal differentiability set,
we 
try to avoid as much as possible constructions that lead to many porosity points inside the set we construct.
More information about porous and $\sigma$-porous sets can be found in the survey \cite{zajicek75}, and further discussion of relations between problems about
differentiability of Lipschitz functions and the theory of porous and
 $\sigma$-porous sets is presented in the recent book \cite{lindenstrausstiserpreiss12}. 

\bigskip

{\textbf{\large{Structure of the paper}}}

\medskip

In Section~\ref{sec:2} we establish a criterion for the universal differentiability property based on the results of \cite{doremaleva1} and \cite{doremaleva3}. Section~\ref{sec:3} is devoted to the construction of the new set 
of upper Minkowski dimension one.  Finally, in Section~\ref{sec:4}, we apply the result established in Section~\ref{sec:2} to prove that our set has the universal differentiability property. 
\section{Differentiability}\label{sec:2}
In this section, we prove a criterion for the universal differentiability property. We begin by defining the two key notions of differentiability of a real valued function $f$ on a Banach space $X$, according to \cite[p.~83]{benyaminilindenstrauss00}.

\begin{definition}
A function $f:X\to\mathbb{R}$ is said to be 
Fr\'echet differentiable at a point $x\in{X}$ if
the limit
\begin{equation*}
f'(x,e)=\lim_{t\to{0}}\frac{f(x+te)-f(x)}{t}
\end{equation*}
exists uniformly in $e\in \overline B(0,1)$
and is a bounded linear map. 
\end{definition}
Next, we formalise what it means for a function $f:X\to\mathbb{R}$ to be Lipschitz \cite[p.~2]{doremaleva2}.
\begin{definition}
A function $f:X\to\mathbb{R}$ is called Lipschitz if there exists $L>0$ such that 
$\left|f(y)-f(x)\right|\leq{L\left\|y-x\right\|_{X}}\mbox{ for all }x,y\in{X}\mbox{ with }y\neq{x}$.
If $f:X\to\mathbb{R}$ is a Lipschitz function then the number 
\begin{equation*}
\mbox{Lip}(f)=\sup\left\{\frac{\left|f(y)-f(x)\right|}{\left\|y-x\right\|_{X}}\mbox{ : }x,y\in{X}\mbox{, }y\neq{x}\right\}
\end{equation*}
is finite and is called the Lipschitz constant of $f$.
\end{definition}
\theoremstyle{plain}
Theorem~\ref{thm:udscrit} is an amalgamation of the results in \cite{doremaleva1} and \cite{doremaleva2}. Roughly speaking, it says that a closed set $S$ has the universal differentiability property if every point in $S$ can be approximated, in a special way, by line segments contained in $S$. We use this theorem in Section~\ref{sec:3} to construct a universal differentiability set with upper Minkowski dimension one.

\begin{theorem}\label{thm:udscrit}
Let $d\geq{2}$.
Suppose that $(\uds_{\lambda})_{\lambda\in[0,1]}$ is a family of closed subsets of $\mathbb{R}^{d}$ satisfying $\uds_{\lambda_{1}}\subseteq{\uds_{\lambda_{2}}}$ whenever $0\leq\lambda_{1}\leq\lambda_{2}\leq{1}$. Suppose further
that for every $\lambda\in[0,1)$, $\psi\in(0,1-\lambda)$ and $\eta\in(0,1)$, there exists
\begin{equation*}
\delta_{1}=\delta_{1}(\lambda,\psi,\eta)>0
\end{equation*}
such that 
whenever $x\in{\uds_{\lambda}}$, $\delta\in(0,\delta_{1})$ and $v_{1},v_{2},v_{3}$ are in the closed unit ball in $\mathbb{R}^{d}$, there exists ${v_{1}}',{v_{2}}',{v_{3}}'\in\mathbb{R}^{d}$ such that $\left\|{v_{i}}'-v_{i}\right\|\leq\eta$ and $[x+\delta{v_{1}}',x+\delta{v_{3}}']\cup[x+\delta{v_{3}}',x+\delta{v_{2}}']\subseteq{\uds_{\lambda+\psi}}$. Then the set
\begin{equation}\label{eq.uds}
S=\overline{\bigcup_{q<1}\uds_{q}}
\end{equation}
is a universal differentiability set. Moreover, $S$ has the property that whenever $y\in{S}$, $\rho>0$ and $g:\mathbb{R}^{d}\to\mathbb{R}$ is a Lipschitz function, there exists $x\in{S}$ such that $\left\|x-y\right\|<\rho$ and $g$ is differentiable at $y$.
\end{theorem}
\begin{proof}[Proof of Theorem~\ref{thm:udscrit}]
Let $y\in{S}$, $\rho>0$ and $g:\mathbb{R}^{d}\to\mathbb{R}$ be a Lipschitz function. We show that there exists a point $x\in{S}$ such that $g$ is differentiable at $x$ and $\left\|x-y\right\|<\rho$. 

We may assume $\mbox{Lip}(g)>0$. Let $H$ be the Hilbert space $\mathbb{R}^{d}$. Set $\lambda_{1}=1$ and fix $\widetilde\lambda_{0}\in(0,1)$ and $y'\in{\uds_{\widetilde\lambda_{0}}}$ such that $\left\|y'-y\right\|<\rho/3$. 
Taking $\delta<\min\{\delta_1,\rho/4\}$
we can find a line segment $[a,b]\subseteq\mathbb{R}^{d}$ with $[a,b]\subseteq{\uds_{\lambda_{0}}\cap{B_{\rho/3}(y')}}$, where $\lambda_0\in(\widetilde\lambda_0,1)$. Then, by Lebesgue's theorem, there exists a point $x_{0}\in[a,b]$ such that the directional derivative $g'(x_{0},e_{0})$ exists, where 
$e_{0}=(b-a)/{\left\|b-a\right\|}$.

Set $f_{0}=g$, $(\mathfrak S,\preceq)=([0,1],\leq)$, $K=25\sqrt{2\mbox{Lip}(g)}$, $\sigma_{0}=\rho/3$, $\mu=\mbox{Lip}(g)$, and apply 
\cite[Theorem 2.7]{doremaleva2}
where $H=\real^d$.

This theorem provides us with
the Lipschitz function $f$, the number $\lambda\in(\lambda_{0},\lambda_{1})$, the point $x\in \uds_\lambda\cap B_{\sigma_0}(x_0)$,
the direction $e\in S^{d-1}$ and, for each $\epsilon>0$, the numbers $\sigma_{\epsilon}>0$ and $\lambda_{\epsilon}\in(\lambda,1)$ 
such that
$f-f_{0}$ is linear with operator norm less than or equal to $\mu$ and
the directional derivative $f'(x,e)>0$ is almost locally maximal in the following sense: 
Whenever 
\begin{enumerate}[(i)]
\item $x'\in F_\epsilon=\uds_{\lambda_{\epsilon}}\cap B_{\sigma_\epsilon}(x)$,
$f'(x',e')\geq{f'(x,e)}$ and
\item for any $t\in\mathbb{R}$
\begin{equation}
|(f(x'+te)-f(x'))-(f(x+te)-f(x))|\leq{K}\sqrt{f'(x',e')-f'(x,e)}|t|,
\end{equation}
\end{enumerate}
then we have $f'(x',e')<f'(x,e)+\epsilon$.

We then verify that
the conditions of \cite[Lemma 2.8]{doremaleva2}
hold for the function $f:\mathbb{R}^{d}\to\mathbb{R}$, the pair $(x,e)$ and the family of sets $F_{\epsilon}\subseteq\mathbb{R}^{d}$ 
defined in (i).
This would imply that the function $f$ is differentiable at $x$.

We already have that the derivative $f'(x,e)$ 
exists and is non-ne\-ga\-ti\-ve. We now verify condition (1) of \cite[Lemma 2.8]{doremaleva2}. Given $\epsilon>0$ and $\eta\in(0,1)$ we put $\psi_{\epsilon}=\lambda_{\epsilon}-\lambda$ and define
\begin{equation*}
\delta_{\star}=
\delta_\star(\epsilon,\eta)=\frac{1}{2}\min\left\{\delta_{1}(\lambda,\eta,\psi_{\epsilon}),\sigma_{\epsilon}\right\},
\end{equation*}
where $\delta_{1}(\lambda,\eta,\psi_{\epsilon})$ is defined in the hypothesis
of Theorem~\ref{thm:udscrit}.
Let $v_{1},v_{2},v_{3}\in{\overline{B}_{1}(0)}\subseteq\mathbb{R}^{d}$ and $\delta\in(0,\delta_{\star})$. Since $0<\delta<\delta_{1}(\lambda,\eta,\psi_{\epsilon})$, there exist, by the hypothesis of Theorem~\ref{thm:udscrit}, points ${v_{1}}',{v_{2}}',{v_{3}}'\in\mathbb{R}^{d}$ with $\left\|{v_{i}}'-v_{i}\right\|\leq\eta$ and
\begin{equation*}
[x+\delta{v_{1}}',x+\delta{v_{3}}']\cup[x+\delta{v_{3}}',x+\delta{v_{2}}']\subseteq{\uds_{\lambda+\psi_{\epsilon}}}=\uds_{\lambda_{\epsilon}}.
\end{equation*} 
Moreover, given that $\delta<\sigma_{\epsilon}$, we have
\begin{equation*}
[x+\delta{v_{1}}',x+\delta{v_{3}}']\cup[x+\delta{v_{3}}',x+\delta{v_{2}}']\subseteq \uds_{\lambda_{\epsilon}}\cap B_{\sigma_{\epsilon}}(x)=F_{\epsilon}.
\end{equation*}
Thus, condition (1) of \cite[Lemma 2.8]{doremaleva2} is verified. Note that $\mbox{Lip}(f)\leq\mbox{Lip}(g)+\mu=2\mbox{Lip}(g)$ so that $25\sqrt{\mbox{Lip}(f)}\leq{K}$. Now, condition (2) of \cite[Lemma 2.8]{doremaleva2} is apparent from the definition of the sets $F_{\epsilon}$ and (ii).

Therefore, by \cite[Lemma 2.8]{doremaleva2} the function $f$ is differentiable at $x$. Since $g-f$ is linear, we conclude that $g$ is also differentiable at $x$. Moreover, we have that $x\in{\uds_{\lambda_{\epsilon}}}\subseteq{S}$ and
\begin{equation*}
\left\|x-y\right\|\leq\left\|x-x_{0}\right\|+\left\|x_{0}-y'\right\|+\left\|y'-y\right\|<\rho/3+\rho/3+\rho/3=\rho.
\end{equation*}
\end{proof}

\section{The Set}\label{sec:3}
We let $d\geq{2}$ and construct a universal differentiability set of upper Minkowski dimension one in $\mathbb{R}^{d}$. There are many equivalent ways of defining the (upper and lower) Minkowski dimension of a bounded subset of $\mathbb{R}^{d}$. In addition to Definition~\ref{def:minkowskidimension} in Section~\ref{sec:1}, several examples can be found in \cite[p.~41-45]{mattila99}. The equivalent definition given below will be most convenient for our use.
By an $\epsilon$-cube, with centre $x\in\mathbb{R}^{d}$, parallel to $e\in{S^{d-1}}$, we mean any subset of $\mathbb{R}^{d}$ of the form
\begin{equation}\label{eq:parcubes}
C(x,\epsilon,e)=\left\{x+\sum_{i=1}^{d}t_{i}e_{i}\mbox{ :
}e_{1}=e\mbox{, }t_{i}\in[-\epsilon,\epsilon]\right\}.
\end{equation}
where $e_{2},\ldots,e_{d}\in{S^{d-1}}$ and $\langle{e_{i},e_{j}}\rangle=0$ whenever $i\neq{j}$.
\begin{definition}\label{minkdimdef}
Given a bounded subset $A$ of $\mathbb{R}^{d}$ and $\epsilon>0$, we denote
by $N_{\epsilon}(A)$ the minimum number of (closed) $\epsilon$-cubes required to
cover $A$. That is, $N_{\epsilon}(A)$ is the smallest integer $n$ for which
there exist $\epsilon$-cubes $C_{1},C_{2},\ldots,C_{n}$ such that
\begin{equation*}
A\subseteq\bigcup_{i=1}^{n}C_{i}.  
\end{equation*}
As in Definition~\ref{def:minkowskidimension}, we define the lower Minkowski dimension of $A$ by \eqref{eq:lowerminkowskidimension} and the upper Minkowski dimension of $A$ by \eqref{eq:upperminkowksidimension}.
\end{definition}
For a point $x\in\mathbb{R}^{d}$ and $w>0$, we shall write
$\overline{B}_{w}(x)$ for the closed ball with centre $x$ and radius $w$
with respect to the Euclidean norm. For a subset $S$ of $\mathbb{R}^{d}$, we
let $\overline{B}_{w}(S)=\bigcup_{x\in S}\overline{B}_{w}(x)$. The
cardinality of a finite set $F$ shall be denoted by $\left|F\right|$. Given a real number $\alpha$, we write $\left[\alpha\right]$ for the integer part of $\alpha$.

Fix two sequences of positive integers $(s_k)$ and $(M_k)$ such that the following conditions are satisfied:
\begin{equation}\label{eq:defsk}
3\le M_k\le s_k;\
M_k, s_k\to\infty;\
\frac{M_k\log s_k}{s_k}\to0
\end{equation}
and
there  exists a sequence $\tilde s_k\ge s_k$ such that
\begin{equation}\label{eq:defsko} 
\frac{\tilde s_{k}-\tilde s_{k-1}}{s_k}\to 0.
\end{equation}

Note that if  the sequence $(\tilde s_{k}-\tilde s_{k-1})_{k\ge2}$
is bounded and $s_k\to\infty$, then \eqref{eq:defsko} is satisfied. 
Hence an example of  sequences $(s_k), (\tilde s_k)$ satisfying 
\eqref{eq:defsk} and \eqref{eq:defsko} 
is 
$\tilde s_k=ak+b$ with $a>0$ and any integer sequence $s_k\to\infty$ such that $3\le s_k\le \tilde s_k$.

We also remark that if $s_k\to\infty$ is such that 
\begin{equation}\label{eq:exsk}
\frac{s_k}{s_{k+1}}\to1,
\end{equation} 
then \eqref{eq:defsko} is satisfied with
$\tilde s_k=s_k$.
Indeed, 
$\displaystyle\frac{\tilde s_k-\tilde s_{k-1}}{s_k}=1-
\frac{s_{k-1}}{s_k}\to0$.

An example of an integer sequence $s_k\to\infty$ satisfying \eqref{eq:exsk} is
$s_k=\max\{3,[F(k)]\}$, where $F(x)$ 
is a linear combination of powers of $x$  such that the 
highest power of $x$ is positive and 
has a positive coefficient. Also, whenever $s_k\to\infty$ satisfies the
condition \eqref{eq:exsk}, the sequence $s_k'=[\log s_k]$ also satisfies this condition and tends to infinity.

Once $(s_k)$ is defined there is much freedom to choose $(M_k)$. 
For example, we may take 
$M_k=\left[s_k^\alpha\right]$ with $\alpha\in(0,1)$ or $M_k=\left[\log s_k\right]$ etc.

Let now 
$\mathcal E_k$ be a maximal $\frac{1}{s_k}$-separated subset of $S^{d-1}$. We therefore get a collection of finite subsets $\mathcal E_k\subseteq S^{d-1}$ such that
\begin{equation}
\label{eq:defEk}
|\mathcal E_k|\le s_k^{2d}
\qquad\mbox{and}\qquad
\forall e\in S^{d-1} \exists\ e'\in \mathcal{E}_k
\mbox{ s.t. }
\|e-e'\|\le\frac1{s_k}.
\end{equation}

\begin{definition}\label{cubes}
Given a line segment $l=x+[a,b]e\subseteq\mathbb{R}^{d}$ and $0<w<\mbox{length}(l)/2$, define $F_{w}(l)$ to be a finite collection of cubes of the form $C(x_{i},w,e)$, defined by \eqref{eq:parcubes}, with $x_{i}\in{l}$,
such that
\begin{equation}\label{numberofcubes}
\overline{B}_{w}(l)\subseteq\bigcup_{C\in \mathcal{F}_{w}(l)} C\mbox{ and }
\left|\mathcal{F}_{w}(l)\right|<\mbox{length}(l)/w.
\end{equation}
\end{definition}
Fix an arbitrary number $\Y\in(1,2)$.
Let $l_{1}$ be a line segment in $\mathbb{R}^{d}$ of length $1$, set $w_{1}=\Y^{-s_{1}}$, $\mathcal{L}_{1}=\left\{l_{1}\right\}$,
$\mathcal{C}_{1}=\mathcal{F}_{w_{1}}(l_{1})$ and
$\mathcal{T}_{1}=\left\{\overline{B}_{w_{1}}(l_{1})\right\}$. We refer to
the collection $\mathcal{L}_{1}$ as `the lines of level~$1$', the
collection $\mathcal{C}_{1}$ as `the cubes of level~$1$' and the collection
$\mathcal{T}_{1}$ as `the tubes of level~$1$'. Note that $\mathcal{C}_{1}$
is a cover of the union of tubes in $\mathcal{T}_{1}$. 
Suppose that $k\geq{2}$, and that we have defined real numbers 
$w_r>0$ and the collections
$\mathcal{L}_{r}$ of lines, $\mathcal{T}_{r}$ of tubes and $\mathcal{C}_{r}$ of cubes of level $r=1,2,\ldots,k-1$ in such a way that
\begin{equation*}
\mathcal{T}_{r}=\left\{\overline{B}_{w_{r}}(l):l\in\mathcal{L}_{r}\right\};\
\mathcal{C}_{r}=\bigcup_{l\in\mathcal{L}_{r}}\mathcal{F}_{w_{r}}(l)
\textrm{ is a cover of }
\bigcup \{T: T\in\mathcal{T}_{r}\}.
\end{equation*}
We now describe how to
construct the lines, cubes and tubes of the $k$th~level.
We start with the definition of the new width. 
Set
\begin{equation}\label{newwidth}
w_{k}=\Y^{-s_{k}}w_{k-1}.
\end{equation}
The collections
$\mathcal{L}_k$ of lines, $\mathcal{T}_k$ of tubes and $\mathcal{C}_k$ of cubes
will be partitioned into exactly $M_{k}+1$ classes and each class will be further partitioned into categories according to the length of the
lines.  We first define the collections of lines, tubes and cubes of level~$k$, class~$0$
respectively by
\begin{equation}\label{linesclass0}
\mathcal{L}_{(k,0)}=\mathcal{L}_{k-1},
\mathcal{T}_{(k,0)}=\left\{\overline{B}_{w_{k}}(l):l\in\mathcal{L}_{(k,0)}\right\}
\textrm{ and }
\mathcal{C}_{(k,0)}=\bigcup_{l\in\mathcal{L}_{(k,0)}}\mathcal{F}_{w_{k}}(l).
\end{equation}
We will say that all lines, tubes and cubes of level $k$, class $0$ have the empty category. From \eqref{linesclass0}
and Definition~\ref{cubes},
we
have that $\mathcal{C}_{(k,0)}$ is a cover of the union of tubes in
$\mathcal{T}_{(k,0)}$. Using \eqref{numberofcubes}, we also have that
\begin{equation}\label{class0coverestimate1}
\left|\mathcal{C}_{(k,0)}\right|=\sum_{l\in\mathcal{L}_{(k,0)}}\left|\mathcal{F}_{w_{k}}(l)\right|\leq\frac{1}{w_{k}}\sum_{l\in\mathcal{L}_{(k,0)}}\mbox{length}(l).
\end{equation}
The intersection of each line $l$ in $\mathcal{L}_{(k,0)}=\mathcal{L}_{k-1}$ with a cube $C\in\mathcal{C}_{k-1}$ is a line segment of length at most $2w_{k-1}$. Moreover, the collection of cubes $\mathcal{C}_{k-1}$ is a cover of the union of lines in $\mathcal{L}_{(k,0)}$. It follows that
\begin{equation}\label{totallengthestimate}
\sum_{l\in\mathcal{L}_{(k,0)}}\mbox{length}(l)\leq{2\left|\mathcal{C}_{k-1}\right|w_{k-1}}.
\end{equation}
Combining \eqref{class0coverestimate1}, \eqref{totallengthestimate} and
\eqref{newwidth} yields
\begin{equation}\label{class0coverestimate2}
\left|\mathcal{C}_{(k,0)}\right|\leq{2\left|\mathcal{C}_{k-1}\right|\Y^{s_{k}}}.
\end{equation}
\begin{definition}\label{def:R}
Given a bounded line segment $l\subseteq\mathbb{R}^{d}$, an integer $j\geq{1}$ with $\mbox{length}(l)\geq \Y^{j} w_{k}/s_k$ and a direction $e\in{S^{d-1}}$, we define a collection of line segments $\mathcal{R}_{l}(j,e)$ as follows: Let 
$\Segment\subseteq{l}$ 
be a maximal $\frac{\Y^j{w_{k}}}{s_{k}}$-separated set and set
\begin{equation*}
\mathcal{R}_{l}(j,e)=\left\{\segment_{x}\mbox{ : }x\in\Segment\right\},
\end{equation*}
where $\segment_{x}$ is the line defined by
\begin{equation}\label{lineform}
\segment_{x}=x+[-1,1]\Y^{j}w_{k}e.
\end{equation}
We note for future reference that
\begin{equation}\label{numberoflines}
\left|\mathcal{R}_{l}(j,e)\right|\leq\frac{2s_{k}\mbox{
length}(l)}{\Y^{j}w_{k}}.
\end{equation}
\end{definition}
For $j\in\left\{1,2,\ldots,s_{k}\right\}$, we define the collection of lines
of level~$k$, class 1, category~$(j)$ by
\begin{equation}\label{linesclass1categoryj} 
\mathcal{L}_{(k,1)}^{(j)}=\bigcup_{l\in\mathcal{L}_{(k,0)}}\bigcup_{e\in\mathcal{E}_{k}}\mathcal{R}_{l}(j,e).
\end{equation}
We emphasise that all the lines in $\mathcal{L}_{(k,1)}^{(j)}$ have the
same length. Indeed, from Definition~\ref{def:R}, we get
\begin{equation}\label{samelength}
\mbox{length}(l)=2{\Y^{j}w_{k}}\mbox{ for all lines
}l\in\mathcal{L}_{(k,1)}^{(j)}.
\end{equation}
From \eqref{numberoflines} and \eqref{linesclass1categoryj}, it
follows that
\begin{equation}\label{lestimate1}
\left|\mathcal{L}_{(k,1)}^{(j)}\right|\leq2w_{k}^{-1}s_{k}\left|\mathcal{E}_{k}\right|\Y^{-j}\sum_{l\in\mathcal{L}_{(k,0)}}\mbox{length}(l).
\end{equation}
Together, \eqref{lestimate1}, \eqref{totallengthestimate} and
\eqref{newwidth} imply
\begin{equation}\label{lestimate2}
\left|\mathcal{L}_{(k,1)}^{(j)}\right|\leq
\left|\mathcal{C}_{k-1}\right|(4s_{k}\left|\mathcal{E}_{k}\right|)\Y^{s_{k}-j}.
\end{equation}
Let
\begin{equation}\label{cubesclass1categoryj}
\mathcal{C}_{(k,1)}^{(j)}=\bigcup_{l\in\mathcal{L}_{(k,1)}^{(j)}}\mathcal{F}_{w_{k}}(l).
\end{equation}
Then, using \eqref{numberofcubes}, \eqref{samelength} and
\eqref{lestimate2} we obtain
\begin{align}
\left|\mathcal{C}_{(k,1)}^{(j)}\right|
&\leq\sum_{l\in\mathcal{L}_{(k,1)}^{(j)}}\left|\mathcal{F}_{w_{k}}(l)\right|
\leq\frac{1}{w_{k}}\sum_{l\in\mathcal{L}_{(k,1)}^{(j)}}\mbox{length}(l)\notag\\
\label{sestimate1}
&\leq\frac{\left|\mathcal{L}_{(k,1)}^{(j)}\right|\times{2}\times\Y^{j}w_{k}}{w_{k}}\leq{8\left|\mathcal{C}_{k-1}\right|s_{k}\left|\mathcal{E}_{k}\right|\Y^{s_{k}}}.
\end{align}
The collection of tubes of level~$k$, class~$1$, category~$(j)$ is defined
by
\begin{equation}\label{tubesclass1categoryj}
\mathcal{T}_{(k,1)}^{(j)}=\left\{\overline{B}_{w_{k}}(l):l\in\mathcal{L}_{(k,1)}^{(j)}\right\}
\end{equation}
From Definition~\ref{cubes}, \eqref{cubesclass1categoryj} and
\eqref{tubesclass1categoryj} it is clear that
$\mathcal{C}_{(k,1)}^{(j)}$ is a cover of the union of tubes in
$\mathcal{T}_{(k,1)}^{(j)}$. We can also use \eqref{tubesclass1categoryj} and \eqref{lestimate2} 
to conclude that
\begin{equation}\label{testimate}
\left|\mathcal{T}_{(k,1)}^{(j)}\right|=\left|\mathcal{L}_{(k,1)}^{(j)}\right|\leq\left|\mathcal{C}_{k-1}\right|(4{s_k}|\mathcal E_k|)\Y^{s_k-j}.
\end{equation}
The collections of lines, cubes and tubes of level~$k$, class~$1$ are now
defined by
\begin{equation*}
\#_{(k,1)}=\bigcup_{j=1}^{s_{k}}\#_{(k,1)}^{(j)},
\end{equation*}
where $\#$ stands for $\mathcal L$, $\mathcal C$ or $\mathcal T$.
Note that $\mathcal{C}_{(k,1)}$ is a cover of the union of tubes in
$\mathcal{T}_{(k,1)}$. Moreover, in view of \eqref{sestimate1}, 
we get
\begin{equation}\label{1cubesestimate}
\left|{\mathcal{C}_{(k,1)}}\right|
\leq\sum_{j=1}^{s_{k}}\left|\mathcal{C}_{(k,1)}^{(j)}\right|
\leq2\left|\mathcal{C}_{k-1}\right|(4s_k^2|\mathcal E_k|)\Y^{s_k}.
\end{equation}
Suppose that $1\leq{m}<M_{k}$ and that we have defined the collections
\begin{equation*}
{\mathcal{L}_{(k,m)}}\mbox{, }{\mathcal{C}_{(k,m)}}\mbox{ and
}{\mathcal{T}_{(k,m)}}
\end{equation*}
of lines, cubes and tubes of level~$k$, class~$m$. Assume that these
collections are partitioned into categories 
\begin{equation*}
\mathcal{L}_{(k,m)}^{(j_{1},\ldots,j_{m})}\mbox{, }
\mathcal{C}_{(k,m)}^{(j_{1},\ldots,j_{m})}\mbox{ and }
\mathcal{T}_{(k,m)}^{(j_{1},\ldots,j_{m})}
\end{equation*}
where the $j_{i}$ are integers satisfying
\begin{equation}\label{lengthsequence}
1\leq{j_{i+1}}\leq{j_{i}}\leq{s_{k}}\mbox{ for all }i.
\end{equation}
Suppose that the following conditions hold.
\begin{align}\label{indi}
\tag{$\mbox{I}_{m}$}
\mbox{length}(l)&=2\Y^{j_m}w_{k}\mbox{ for all lines
}l\mbox{ in }\mathcal{L}_{(k,m)}^{(j_{1},\ldots,j_{m})},\\
\label{indii}
\tag{$\mbox{II}_{m}$}
\left|\mathcal{L}_{(k,m)}^{(j_{1},\ldots,j_{m})}\right|
&\leq
\left|\mathcal{C}_{k-1}\right|
(4s_{k}|\mathcal E_k|)^m\Y^{s_{k}-j_{m}},\\
\label{indiii}
\tag{$\mbox{III}_{m}$}
\mathcal{T}_{(k,m)}^{(j_{1},\ldots,j_{m})}
&=\left\{\overline B_{w_{k}}(l):l\in\mathcal{L}_{(k,m)}^{(j_{1},\ldots,j_{m})}\right\},\\
\label{indiv}
\tag{$\mbox{IV}_{m}$}
\mathcal{C}_{(k,m)}^{(j_{1},\ldots,j_{m})}
&=\bigcup_{l\in\mathcal{L}_{(k,m)}^{(j_{1},\ldots,j_{m})}}\mathcal{F}_{w_{k}}(l),\\
\label{indvi}
\tag{$\mbox{V}_{m}$}
\left|\mathcal{C}_{(k,m)}^{(j_{1},\ldots,j_{m})}\right|
&\leq
2\left|\mathcal{C}_{k-1}\right|(4s_k|\mathcal E_k|)^m\Y^{s_{k}}.
\end{align}
For an integer sequence $(j_{1},\ldots,j_{m},j_{m+1})$ satisfying
\eqref{lengthsequence} we define the collection of lines of level~$k$,
class~$(m+1)$, category~$(j_{1},\ldots,j_{m+1})$ by
\begin{equation}\label{linesm}
\mathcal{L}_{(k,m+1)}^{(j_{1},\ldots,j_{m+1})}=\bigcup_{l\in
\mathcal{L}_{(k,m)}^{(j_{1},\ldots,j_{m})}}\left(\bigcup_{e\in\mathcal{E}_{k}}\mathcal{R}_{l}(j_{m+1},e)\right).
\end{equation}
Note that every line in the collection $\mathcal{L}_{(k,m+1)}^{(j_{1},\ldots,j_{m+1})}$ has the same length. In fact, by Definition~\ref{def:R} we have that $(\mbox{I}_{m+1})$ is satisfied. Combining \eqref{numberoflines}, \eqref{indi} and \eqref{indii} we deduce the following:
\begin{align}
\left|\mathcal{L}_{(k,m+1)}^{(j_{1},\ldots,j_{m+1})}\right|
&\leq{2s_{k}\left|\mathcal{E}_{k}\right|w_{k}^{-1}\Y^{-j_{m+1}}\sum_{l\in
\mathcal{L}_{(k,m)}^{(j_{1},\ldots,j_{m})}}\mbox{length}(l)}\notag\\
\label{Lestimate}
&\leq
\left|\mathcal{C}_{k-1}\right|(4s_k|\mathcal E_k|)^{m+1}\times\Y^{s_k-j_{m+1}}.
\end{align}
Thus, $(\mbox{II}_{m+1})$ is satisfied. We define the collection of tubes
and cubes of level~$k$ and class~${m+1}$, category~$(j_{1},\ldots,j_{m+1})$ by 
($\mbox{III}_{m+1}$)
and ($\mbox{IV}_{m+1}$). Using \eqref{numberofcubes}, $(\mbox{I}_{m+1})$ and \eqref{Lestimate} we obtain
\begin{multline*}
\left|\mathcal{C}_{(k,m+1)}^{(j_{1},\ldots,j_{m+1})}\right|
\leq\sum_{l\in\mathcal{L}_{(k,m+1)}^{(j_{1},\ldots,j_{m+1})}}\left|\mathcal{F}_{w_{k}}(l)\right|
=\frac{1}{w_{k}}\sum_{l\in\mathcal{L}_{(k,m+1)}^{(j_{1},\ldots,j_{m+1})}}\mbox{length}(l)\\
\leq\frac{1}{w_{k}}\times{\left|\mathcal{L}_{(k,m+1)}^{(j_{1},\ldots,j_{m+1})}\right|}\times{2}\times{\Y^{j_{m+1}}w_{k}}
\leq{2\left|\mathcal{C}_{k-1}\right|{(4s_k|\mathcal E_k|)}^{m+1}\Y^{s_{k}}},
\end{multline*}
and this verifies $(\mbox{V}_{m+1})$. The collections of lines, tubes and cubes of level~$k$,
class~$m+1$ are given by
\begin{equation*}
\#_{(k,m+1)}=\bigcup_{s_k\ge j_1\ge\dots\ge j_{m+1}\ge1}\#_{(k,m+1)}^{(j_{1},\ldots,j_{m+1})},
\end{equation*}
where $\#$ stands for $\mathcal L$, $\mathcal T$ or $\mathcal C$.

Note that $\mathcal{C}_{(k,m+1)}$ is a cover of the union of tubes in
$\mathcal{T}_{(k,m+1)}$. Moreover, in view of 
($\mbox{V}_{m+1}$)
and
\eqref{lengthsequence} we have
\begin{multline}\label{msestimate}
\left|\mathcal{C}_{(k,m+1)}\right|
\leq
\sum_{(j_{1},\ldots,j_{m+1})}\left|\mathcal{C}_{(k,m+1)}^{(j_{1},\ldots,j_{m+1})}\right|
\\
\leq
2\left|\mathcal{C}_{k-1}\right|(4s_k|\mathcal E_k|)^{m+1}s_k^{m+1}\Y^{s_k}
=2\left|\mathcal{C}_{k-1}\right|(4s_k^2|\mathcal E_k|)^{m+1}\Y^{s_k},
\end{multline}
and this generalises \eqref{1cubesestimate} for arbitrary $0\le m<M_k$.
Finally, the collections of lines, tubes and cubes of level~$k$ are given by
\begin{equation}\label{levelktubes}
\mathcal{L}_{k}=\bigcup_{m=0}^{M_{k}}\mathcal{L}_{(k,m)}\mbox{, }\mathcal{T}_{k}=\bigcup_{m=0}^{M_{k}}\mathcal{T}_{(k,m)}\mbox{ and }\mathcal{C}_{k}=\bigcup_{m=0}^{M_{k}}{\mathcal{C}_{(k,m)}}.
\end{equation}
Note that $\mathcal{C}_{k}$ is a cover of the union of tubes in $\mathcal{T}_{k}$. Moreover, using 
\eqref{class0coverestimate2} and \eqref{msestimate} we get
\begin{equation}\label{eq:pf-cond-boxes}
\left|\mathcal{C}_{k}\right|\leq
2(M_k+1)\left|\mathcal{C}_{k-1}\right|(4s_k^2|\mathcal E_k|)^{M_k}\Y^{s_k}.
\end{equation}
The construction of the lines, tubes and cubes of all levels is now complete.

We now define a collection of closed sets $(\uds_{\lambda})_{\lambda\in[0,1]}$. Eventually, we will use these sets to form 
a compact universal differentiability set $S$ with upper Minkowski dimension one, defined by \eqref{eq.uds}. To do this, we follow a method invented by Dor\'{e} and Maleva and used in \cite{doremaleva1}, \cite{doremaleva2} and \cite{doremaleva3}. The sets 
$\uds_\lambda$ are defined similarly to the sets $(T_{\lambda})$ in 
\cite[Definition~2.3]{doremaleva2}.
\begin{definition}\label{finalset}
For $\lambda\in[0,1]$ we let
\begin{equation}\label{set}
\uds_{\lambda}=\bigcap_{k=1}^{\infty}\left(\bigcup_{0\leq{m_{k}}\leq{\lambda{M_{k}}}}\left(\bigcup_{l\in\mathcal{L}_{(k,m_{k})}}\overline{B}_{\lambda{w_{k}}}(l)\right)\right).
\end{equation}
\end{definition}
We emphasise that the single line segment $l_{1}$ of level $1$ is contained in the set ${\uds_{\lambda}}$ for every $\lambda\in[0,1]$. Hence, every $\uds_{\lambda}$ is non-empty. Note also that $\uds_{\lambda_{1}}\subseteq{\uds_{\lambda_{2}}}$ whenever $0\leq\lambda_{1}\leq\lambda_{2}\leq{1}$. Finally, since the unions in \eqref{set} are finite, it is clear that the sets $\uds_{\lambda}$ are closed.
\begin{lemma}\label{lemma:dimension}
For $\lambda\in[0,1]$, the set $\uds_{\lambda}$ has upper Minkowski dimension one.
\end{lemma}
\begin{proof}
For any $\lambda\in[0,1]$ we have that $\uds_{\lambda}$ contains a line segment. Hence, each of the sets $\uds_{\lambda}$ has upper Mikowski dimension at least one. We also have $\uds_{\lambda}\subseteq{\uds_{1}}$ for all $\lambda\in[0,1]$. Therefore, to complete the proof, it suffices to show that the set $\uds_{1}$ has upper Minkowski dimension one. From \eqref{levelktubes}, it is clear that
for all $k\geq{1}$ and $0\leq{m}\leq{M_{k}}$
\begin{equation*}
\bigcup_{l\in\mathcal{L}_{(k,m)}}\overline{B}_{w_{k}}(l)\subseteq\bigcup_{l\in\mathcal{L}_{k}}\overline{B}_{w_{k}}(l)
=\bigcup_{T\in\mathcal{T}_{k}} T.
\end{equation*}
We conclude, using Definition~\ref{finalset}, that for all $k\geq{1}$
\begin{equation}\label{contained}
\uds_{1}\subseteq\bigcup_{T\in\mathcal{T}_{k}} T.
\end{equation}
Let $k\geq{1}$. Recall that $\mathcal{C}_{k}$ is a finite collection of $w_{k}$-cubes which cover the union of tubes $T$ in $\mathcal{T}_{k}$. Therefore, in view of \eqref{contained}, we have that $\mathcal{C}_{k}$ is also a cover of $\uds_{1}$.  By Definition~\ref{minkdimdef}, this means
\begin{equation}\label{eq:Nwk}
N_{w_{k}}(\uds_{1})\leq\left|\mathcal{C}_{k}\right|\mbox{ for all }k\geq{1}.
\end{equation}
Fix an arbitrary $p\in(1,2)$. We complete the proof of this lemma by showing that $\overline{\dim}_{M}(\uds_{1})\leq{p}$. 

For this fixed $1<p<2$ we claim that 
the sequence 
$\left|\mathcal{C}_{k}\right|{w_{k}}^{p}\Y^{ps_k}$ is bounded, i.e.
there exist 
$H>0$ such that 
\begin{equation}
\label{eq:claim}
|\mathcal{C}_{k}|{w_{k}}^{p}\Y^{ps_k}\le H
\quad
\forall k\ge1.
\end{equation}
Assume that the claim is valid. 
Fix an arbitrary $w\in(0,w_1)$. There exists an integer $k\ge 1$ such that $w_{k+1}\leq{w}<w_{k}$. This implies $N_{w}(\uds_{1})\leq{N_{w_{k+1}}(\uds_{1})}$ so that
\begin{equation}\label{eq:defRk}
N_{w}(\uds_{1})w^{p}\leq{N_{w_{k+1}}(\uds_{1})w_{k}^p}=N_{w_{k+1}}(\uds_{1})w_{k+1}^{p}\Y^{ps_{k+1}}\le H.
\end{equation}  
Hence, the sequence $N_{w}(\uds_{1})w^{p}$ is uniformly bounded from above by a fixed constant $H$. Since this is true 
for any arbitrarily small ${w\in(0,w_{1})}$, we conclude that $\overline{\dim}_{M}(\uds_{1})\leq{p}$.

It only remains to establish the  claim \eqref{eq:claim}. 
We prove a more general statement, namely, that the sequence
$|\mathcal C_k|w_k^p\Y^{p \tilde s_k}$ tends to zero.

From \eqref{eq:defsk}, it follows that $s_{k}\geq M_{k}+1\geq 4$ for sufficiently large $k$. Using this, together with \eqref{eq:pf-cond-boxes}, \eqref{eq:defEk} and \eqref{newwidth}
we obtain
\begin{align}
\frac{|\mathcal C_k|w_k^p\Y^{p \tilde s_k}}{|\mathcal C_{k-1}|w_{k-1}^p\Y^{p \tilde s_{k-1}}}
&\le 
2(M_k+1)(4s_k^2|\mathcal E_k|)^{M_{k}}\Y^{-(p-1)s_{k}}\Y^{p(\tilde s_k-\tilde s_{k-1})}\notag\\
&\le
s_k^2s_k^{(3+2d)M_k}\Y^{-(p-1)s_k}\times
\Y^{p(\tilde s_k-\tilde s_{k-1})}\notag\\
\label{eq:estsk}&\le
\Y^{(p-1) s_k/2}\times
\Y^{-(p-1)s_k}\times
\Y^{p(\tilde s_k-\tilde s_{k-1})}
\end{align}
for sufficiently large $k$. 
The latter inequality follows from
\[\frac{M_k\log s_k}{s_k}<\frac{(p-1)\log\Y}{2(5+2d)},\] 
which is true for sufficiently large $k$.
We then see that the product of the three terms in \eqref{eq:estsk} 
tends to zero as $k\to\infty$, since \eqref{eq:defsko} implies that
\[ 
p(\tilde s_k-\tilde s_{k-1})<(p-1) s_k/4
\]
for $k$ sufficiently large. 
\end{proof}
\section{Main Result}\label{sec:4}  
The objective of this section is to prove Theorem~\ref{thm:main} which guarantees, in every finite dimensional space,
the existence of a compact universal differentiability set of upper and lower Minkowski dimension one. We should note that one cannot achieve a better Minkowski dimension as any universal differentiability set has Hausdorff dimension at least one \cite[Lemma~1.5]{doremaleva2}, hence Minkowski dimension of a universal differentiability set should be at least one. We also note that we will always assume $d\ge2$ as the case $d=1$ is trivial, we can simply take $S=[0,1]$. 
 
We will first need to establish several lemmas. The statements we prove typically concern a line $l$ of level~$k$, class~$m$, category~$(j_{1},\ldots,j_{m})$ where $0\leq{m}\leq{M_{k}}$. When $m=0$, we interpret the category $(j_{1},\ldots,j_{m})$ as the empty category and assume $j\leq{j_{m}}$ for all integers $j$.
\begin{lemma}\label{basic}
Let $k\geq{2}$, $0\leq{m}<M_{k}$ and
and $l\in\mathcal L_{(k,m)}^{(j_1,\dots,j_m)}$.
Let $e\in{\mathcal{E}_{k}}$ and
$1\leq j_{m+1}\leq j_m\leq s_k$.
If $x\in l$, then there exists $x'\in l$ such that 
$\left\|x'-x\right\|\leq\Y^{j_{m+1}}w_{k}/s_{k}$
and 
\begin{equation*}
l'=x'+[-1,1]Q^{j_{m+1}}w_{k}e\in
\mathcal L_{(k,m+1)}^{(j_{1},\ldots,j_{m},j_{m+1})}
\end{equation*}
\end{lemma}
\begin{proof}
We observe that 
by definition, the collection $\mathcal R_{l}(j_{m+1},e)$ has an element $l'$ satisfying the conclusions of this lemma.
\end{proof}
\begin{lemma}\label{approx}
Let $k\geq{2}$ and suppose $1\leq{m}\leq{M_{k}}$. Let 
$x\in l\in\mathcal L_{(k,m)}^{(j_{1},\ldots,j_{m})}$ and
$i_{m}$ be an integer with $j_{m}<i_{m}\leq{s_{k}}$. Then there exists an integer sequence
$s_{k}\geq{i_{1}}\geq\ldots\geq{i_{m-1}}\geq{i_{m}}$ and a line 
$l'\in \mathcal L_{(k,m)}^{(i_{1},\ldots,i_{m})}$, 
such that $l'$ is
parallel to $l$ and there exists a point $x'\in{l'}$ with
$\left\|x'-x\right\|\leq\frac{m\times{\Y^{i_{m}}w_{k}}}{s_{k}}$.
\end{lemma}
\begin{proof} 
Suppose that either 
\begin{enumerate}[(i)]
\item $n=1$, or
\item $2\leq{n}\leq{M_{k}}$ and the statement of Lemma~\ref{approx} holds for integers $m=1,\ldots,n-1$.
\end{enumerate}  
We prove that in both cases, the statement of Lemma~\ref{approx} holds for $m=n$. The proof will then be complete, by induction. 

Let the line $l$, integers $j_{1},\ldots,j_{n},i_{n}$ and point $x\in{l}$ be given by the hypothesis of Lemma~\ref{approx} when we set $m=n$. Let $e\in\mathcal{E}_{k}$ be the direction of $l$. By \eqref{linesclass1categoryj} in case (i), or \eqref{linesm} in case (ii), there
exists a line $l^{(n-1)}$ of level~$k$,
class~$n-1$, category~$(j_{1},\ldots,j_{n-1})$ such that the line $l$ belongs to the collection
$\mathcal{R}_{l^{(n-1)}}(j_{n},e)$.

Let the line segment $l^{(n-1)}$ be parallel to $f^{(n-1)}\in{{S}^{1}}$. By
Definition~\ref{def:R}, the line $l$ has the form
\begin{equation*}
l=z+[-1,1]\Y^{j_{n}}w_{k}e
\end{equation*}
where $z\in l^{(n-1)}$. Therefore, we may write
\begin{equation}\label{z}
z=x+\beta e,
\end{equation}
where
\begin{equation}\label{newdelta1}
\left|\beta \right|\leq\Y^{j_n}w_k.
\end{equation}
We now distinguish between two cases. First suppose that
$i_{n}\leq{j_{n-1}}$. Note that this is certainly the case if $n=1$. Setting $i_{a}=j_{a}$ for $a=1,\ldots,n-1$, we get
that $s_{k}\geq{i_{1}}\geq\ldots\geq{i_{n-2}}\geq{i_{n-1}}\ge i_n$. The 
line $l^{(n-1)}\in\mathcal{L}_{(k,n-1)}^{(i_{1},\ldots,i_{n-1})}$, the direction $e\in\mathcal{E}_{k}$, the integer $i_{n}$ and the point $z\in{l^{(n-1)}}$ now satisfy the conditions of Lemma~\ref{basic}. Hence
there is a line $l'$ of level~$k$,
class~$n$, category~$(i_{1},\ldots,i_{n})$ and a point $z'$ with
\begin{equation}\label{z'toz}
\left\|z'-z\right\|\leq\frac{\Y^{i_{n}}w_{k}}{s_{k}},
\end{equation}
such that the line segment $l'$ is given by
\begin{equation*}\label{linel'}
l'=z'+[-1,1]\Y^{i_{n}}w_{k}e.
\end{equation*}
Finally, set
\begin{equation*}\label{x'case1}
x'=z'-\beta e,
\end{equation*}
so that $x'\in{l'}$, using 
\eqref{newdelta1}. We deduce,
using 
\eqref{z'toz} and \eqref{z} that
$\left\|x'-x\right\|\leq\frac{\Y^{i_{n}}w_{k}}{s_{k}}\leq\frac{n\times{\Y^{i_{n}}w_{k}}}{s_{k}}$. This completes
the proof for the case $i_{n}\leq{j_{n-1}}$.
\newline\newline
Now suppose that $i_{n}>j_{n-1}$. In this situation, we must be in case (ii). 
We set $i_{n-1}=i_{n}>j_{n-1}$. The
conditions of Lemma~\ref{approx} are now readily verified for $z\in l^{(n-1)}\in\mathcal{L}_{(k,n-1)}^{(j_{1},\ldots,j_{n-1})}$,
and the integer $i_{n-1}$. 
Therefore, by (ii) and Lemma~\ref{approx}, there exists an integer sequence $s_{k}\geq{i_{1}}\geq\ldots\geq{i_{n-2}}\geq{i_{n-1}}$ and a line $l''\in\mathcal L_{(k,n-1)}^{(i_{1},\ldots,i_{n-1})}$ 
such that $l''$
is parallel to $l^{(n-1)}$ and there exists a point $y''\in{l''}$ such that
\begin{equation}\label{y''}
\left\|y''-z\right\|\leq\frac{(n-1)\times{\Y^{i_{n-1}}w_{k}}}{s_{k}}.
\end{equation}
The conditions of Lemma~\ref{basic} are now readily verified for the 
line $l''\in\mathcal{L}_{(k,n-1)}^{(i_{1},\ldots,i_{n-1})}$, the direction $e\in\mathcal{E}_{k}$, the integer $i_{n}$ and the point $y''\in{l''}$. Hence
there exists a line 
$l'\in\mathcal L_{(k,n)}^{(i_{1},\ldots,i_{n})}$ 
and a point $y'\in{l'}$ such that
\begin{equation}\label{y'}
\left\|y'-y''\right\|\leq\Y^{i_{n}}w_{k}/s_{k}.
\end{equation}
and the line $l'$ is given by
\begin{equation*}\label{newl'}
l'=y'+[-1,1]\Y^{i_{n}}w_{k}e.
\end{equation*}
We set
\begin{equation*}\label{x'}
x'=y'-\beta e.
\end{equation*}
Using 
\eqref{newdelta1} and $i_n>j_n$ 
we get 
that $x'\in{l'}$. Moreover, using \eqref{z}, \eqref{y''} and \eqref{y'}, 
we obtain $\left\|x'-x\right\|\leq\frac{n\times{\Y^{i_{n}}w_{k}}}{s_{k}}$.
\end{proof}
\begin{lemma}\label{lemma:c1}
Let $\lambda\in[0,1)$, 
$\psi\in\bigl(0,1-\lambda\bigr)$ 
and suppose that $x\in{\uds_{\lambda}}$. Suppose that the integer $n$ and number $\delta>0$ satisfy 
\begin{equation}\label{eq:tnd}
\psi{\Y^{t-1}w_{n}}<\delta\leq\psi{\Y^{t}w_{n}}
\mbox{ and }\frac{1}{s_{n}}\leq\psi
\end{equation}
where $t\in\left\{0,1,\ldots,s_{n}-1\right\}$. 

Let $f\in\mathcal{E}_{n}$ and suppose that 
$y\in l\in\mathcal L_{(n,\myr)}^{(h_{1},\ldots,h_{\myr})}$, 
where 
\begin{equation}\label{eq:ch}
\myr\leq(\lambda+\psi)M_{n}-2\mbox{, }h_{\myr}=t+1.
\end{equation}
Then there exists a line 
$l'\in\mathcal L_{(n,1+\myr)}^{(h_{1},\ldots,h_{\myr},t+1)}$ 
and a point $y'\in l$ such that 
\begin{align}
\left\|y'-y\right\|&\leq\frac{\Y}{\psi{s_{n}}}\times\delta, \label{eq:y-y'} \\
l'&=y'+[-1,1]\Y^{t+1}w_{n}f \label{eq:lineeq},\\
\text{ and }
y'+[-1,1]\ \tau f
&\subseteq{\uds_{\lambda+\psi}\cap{l'}},\text{ whenever }\notag\\
\label{eq:tau}
0\leq
\tau&\leq(\Y-\frac{\Y}{\psi s_{n}})\delta-\left\|y-x\right\|.
\end{align}
\end{lemma}
\begin{proof}
Choose a sequence of integers $(m_k)_{k\geq1}$ with $0\leq m_k\leq\lambda M_k$, and a sequence $(l_{k})_{k\geq{1}}$ of line segments such that \mbox{$l_{k}\in\mathcal{L}_{(k,m_{k})}$} is a line of level~$k$, class~$m_{k}$ and 
\begin{equation}\label{eq:branch}
x\in\bigcap_{k=1}^{\infty}\overline{B}_{\lambda{w_{k}}}(l_{k}).
\end{equation}
Note that $\Y\delta<\psi w_n\Y^{t+1}\le\psi w_n\Y^{s_n}=\psi w_{n-1}\le
\psi{w_{k}}$ for all $k\leq{n-1}$. This,
together with \eqref{eq:branch}, implies that
\begin{equation}\label{initialtubes}
\overline{B}_{\Y\delta}(x)\subseteq\overline{B}_{(\lambda+\psi)w_{k}}(l_{k})\mbox{
for }1\leq{k}\leq{n-1}.
\end{equation}
Now, the line $l\in\mathcal L_{(n,\myr)}^{(h_{1},\ldots,h_{\myr})}$, the direction $f\in\mathcal{E}_{n}$, the integer $t+1$ and the point $y\in l$
satisfy the conditions of Lemma~\ref{basic}. Therefore,
there exists a line $l'$ of level $n$, class $1+\myr$, category $(h_{1},\ldots,h_{\myr},t+1)$ and a point $y'\in l'$ such that \eqref{eq:lineeq} holds and
\begin{equation}\label{eq:xerror}
\left\|y'-y\right\|\leq\frac{\Y^{t}w_{n}}{s_{n}}=\frac{\Y}{\psi s_{n}}\times{\psi \Y^{t-1}w_{n}}\leq\frac{\Y}{\psi s_{n}}\times{\delta}.
\end{equation}
Recall that $l'$ is a line of level~$n$. Hence, from 
\eqref{linesclass0} we have that $l'$ is a line of level~$k$, class~$0$ for
all $k\geq{n+1}$. We now set
\begin{equation}\label{l's}
{l_{k}}'=l'\mbox{ for all }k\geq{n}
\mbox{ and }{l_{k}}'=l_{k}
\mbox{ for
}1\leq{k}\leq{n-1}.
\end{equation}
Then for each $k\geq{1}$, we have that ${l_{k}}'$ is a line of level~$k$ class~${m_{k}}'$ where
\begin{equation*}
{m_{k}}'= \begin{cases} m_{k} & \mbox{ if }1\leq{k}\leq{n-1}, \\
1+\myr & \mbox{ if }k=n \\
0 & \mbox{ if }k\geq{n+1}.
\end{cases}
\end{equation*}
From $m_{k}\leq\lambda{M_{k}}$ and \eqref{eq:ch} we have that $0\leq{m_{k}}'\leq{(\lambda+\psi)M_{k}}$ for all $k$. Hence, by Definition~\ref{finalset},
\begin{equation}\label{branch2}
\bigcap_{k=1}^{\infty}\overline{B}_{(\lambda+\psi)w_{k}}(l_k')\subseteq{\uds_{\lambda+\psi}}.
\end{equation}
Suppose $\tau$ is a real number satisfying \eqref{eq:tau}
(by \eqref{eq:tnd} we have that $\psi-1/s_n$ is non-negative). As $\psi<1$,
\begin{equation*}
0\le\tau\leq\Y^{t+1}\left(\psi-\frac{1}{s_{n}}\right)w_{n}
\leq
\Y^{t+1}w_{n}.
\end{equation*}
Hence 
$y'+[-1,1]\tau f\subseteq{l'}$ by \eqref{eq:lineeq}.

From \eqref{initialtubes}, \eqref{eq:xerror} and \eqref{eq:tau} we have that
for all $1\leq{k}\leq{n-1}$
\begin{equation*}
y'+[-1,1]\tau f
\subseteq{l'}\cap\overline{B}_{\Y\delta}(x)
\subseteq{l'}\cap\overline{B}_{(\lambda+\psi)w_k}(l_k).
\end{equation*}
Putting this together with \eqref{l's}, we conclude that 
\[
y'+[-1,1]\tau f
\subseteq
\uds_{\lambda+\psi}\cap{l'}
\]
as $l'\subseteq \overline{B}_{(\lambda+\psi)w_{k}}(l')=
\overline{B}_{(\lambda+\psi)w_{k}}(l_{k})
$ for all $k\ge n$.
\end{proof}
The next Lemma represents the crucial step towards our main result Theorem~\ref{thm:main}. 
\begin{lemma}\label{lemma:crucial}
Let $\lambda\in(0,1)$, $\psi\in\bigl(0,1-\lambda\bigr)$ and
$\eta\in(0,1/4)$. Then there exists a real number
\begin{equation}
\delta_{0}=\delta_{0}(\lambda,\psi,\eta)>0
\end{equation}
such that for any $x\in{\uds_{\lambda}}$, $e\in{S^{d-1}}$ and $\delta\in(0,\delta_{0})$, there exists 
$e'\in{S^{d-1}}$, integers $n,t$ and a pair $(x',l')$, consisting of a point and 
a straight line segment, with $x'\in l'\in\mathcal L_{n,\myr}^{(h_{1},\ldots,h_{\myr})}$, satisfying the following properties.

\begin{enumerate}[(i)]
\item Condition \eqref{eq:tnd} of Lemma~\ref{lemma:c1} is satisfied;
\item Condition
\begin{equation}\label{eq:ch:str}
\myr\leq(\lambda+\psi)M_{n}-4\mbox{, }h_{\myr}=t+1
\end{equation}
is satisfied (a stronger version of \eqref{eq:ch});
\item $\left\|x'-x\right\|\leq\eta\delta$, $\left\|e'-e\right\|\leq\eta$ and 
\begin{equation}\label{eq:x+deltae}
x'+[-1,1]\delta e'\subseteq{\uds_{\lambda+\psi}}\cap{l'}.
\end{equation}
Moreover, $\delta_0$ can be chosen to be independent of $\Y\in(1,2)$.
\end{enumerate}
\end{lemma}
\begin{proof}
We will find $\delta_0'=\delta_0'(\lambda,\psi,\eta)$ 
such that for any $x\in \uds_{\lambda}$, $e\in S^{d-1}$ and $\delta\in(0,\delta_{0}')$, conclusions (i), (ii) and (iii) of Lemma \ref{lemma:crucial} are valid when
\eqref{eq:x+deltae} is replaced by a weaker statement
\begin{equation}\label{eq:x+deltae:we}
x'+[-1,1]\frac\delta2{e'}\subseteq{\uds_{\lambda+\psi}}\cap{l'}.
\end{equation}
Then, defining $\delta_0=
\displaystyle\frac12\delta_0'(\lambda,\psi,\eta/2)$, we will get that the conclusion of this lemma including
\eqref{eq:x+deltae}  is satisfied.

Since $(w_{k})_{k\geq{1}}$ is strictly decreasing, and 
the sequences~$(s_{k})$ and $(M_{k})$ satisfy $s_k,M_k\to\infty$, $s_k/M_k\to0$
by \eqref{eq:defsk},
we may choose 
$\displaystyle\delta_0'\in(0,\frac{\psi}{2}w_{1})$ 
small enough so that
whenever $\psi{w_{k}}\leq2\delta_0'$ we have
\[\frac{1}{s_{k}}\leq\min\left\{\eta,\psi\right\},
\qquad
\frac{M_{k}+4}{s_{k}}\leq\frac{\eta\psi}{4},
\qquad
\psi M_k\geq 6.\]
As $Q\in(1,2)$, this implies that
whenever $\psi{w_{k}}\leq\Y\delta_0'$ we have
\begin{equation}\label{eq:delta0}
\frac{1}{s_{k}}\leq\min\left\{\eta,\psi\right\},
\qquad
\frac{M_{k}+4}{s_{k}}\leq\frac{\eta\psi}{\Y^2},
\qquad
\psi M_k\geq 6.
\end{equation}

Let $x\in{\uds_{\lambda}}$ and fix $\delta\in(0,\delta_0')$. Choose a sequence of integers $(m_k)_{k\geq1}$ with $0\leq m_k\leq\lambda M_k$ and a sequence $(l_{k})_{k\geq{1}}$ of line segments such that \mbox{$l_{k}\in\mathcal{L}_{(k,m_{k})}$} is a line of level~$k$, class~$m_{k}$ and 
\begin{equation*}
x\in\bigcap_{k=1}^{\infty}\overline{B}_{\lambda{w_{k}}}(l_{k}).
\end{equation*}
Note that $\Y\delta<\psi w_1$ as $\Y<2$.
Since $w_k\to0$, there is 
a unique natural number $n\ge2$ satisfying 
\begin{equation}\label{eq:n}
\psi{w_{n}}\leq{\Y\delta}<\psi{w_{n-1}}.
\end{equation}
Using \eqref{eq:n} and $w_{n-1}=\Y^{s_{n}}w_{n}$, we can find $t\in\left\{0,1,\ldots,s_{n}-1\right\}$ satisfying
\begin{equation*}
\psi\Y^{t}w_{n}\leq{\Y\delta}<\psi\Y^{t+1}w_{n}.
\end{equation*}
Further, from \eqref{eq:delta0}, $\delta\in(0,\delta_{0})$ and \eqref{eq:n} we have that $1/s_{n}\leq\psi$. Hence $\delta$, $n$ and $t$ satisfy \eqref{eq:tnd}. By \eqref{eq:defEk}
there exists a direction $e'\in\mathcal{E}_{n}$ such that $\left\|e'-e\right\|\leq 1/s_{n}$, whilst $1/s_{n}\leq \eta$ follows from \eqref{eq:delta0}, $\delta\in(0,\delta_0')$ and \eqref{eq:n}. Hence, we have $\left\|e'-e\right\|\leq \eta$ as required.

Note that $\overline{B}_{w_{n}}({l_{n}})$ is a tube of level~$n$, class~$m_{n}$, containing the point $x$. Let the line ${l_{n}}$ have category~$(j_{1},\ldots,j_{m_{n}})$. 
We can write $x=z+\alpha g$ where 
$z\in l_n$, $g\in{S^{d-1}}$ and
$\alpha\in[0,\lambda w_{n}]$.
Next, using \eqref{eq:defEk}, 
pick $g'\in\mathcal{E}_{n}$ such that $\left\|g'-g\right\|\leq 1/s_{n}$. 
Apply now Lemma~\ref{basic} to $z\in l_n$ to find 
a line 
\[
l'''=z'+[-1,1]\Y w_ng'\in\mathcal L_{(k,1+m_n)}^{(j_{1},\ldots,j_{m_{n}},1)},
\] 
where
$z'\in l_n$ and $\|z'-z\|\le \Y w_n/s_n$. Let
$x'''=z'+\alpha g'$; then  
\begin{multline}\label{eq:x'''}
\left\|x'''-x\right\|\leq
\|z'-z\|+\alpha\|g'-g\|\\
\le
\frac{2\Y w_{n}}{s_{n}}\leq\frac{2\Y}{\psi s_{n}}\times\psi \Y^{t}w_{n}\leq\frac{2\Y^{2}}{\psi s_{n}}\times\delta.
\end{multline}

From \eqref{eq:delta0}, $\delta\in(0,\delta_0')$ and \eqref{eq:n}, we have $\psi{M_{n}}\geq{6}$. In particular, 
\begin{equation*}\label{eq:classbound}
m_{n}+2\leq \lambda{M_n}+2\leq(\lambda+\psi)M_{n}-4,
\end{equation*} 
and \eqref{eq:ch:str} is satisfied when $\myr=m_{n}+2$ and $h_{\myr}=t+1$.

We will now show that there exists a line $l'$ of level~$n$, class~$2+m_{n}$ category~$(j_{1},\ldots,j_{1+m_{n}},t+1)$, and a point $x'\in{l'}$ such that 
\begin{equation}\label{x'-x'''}
\left\|x'-x'''\right\|\leq\frac{(m_{n}+2)\Y^{2}}{\psi s_n}\times\delta
\mbox{ and }
x'+[-1,1]\frac\delta2{e'}\subseteq{\uds_{\lambda+\psi}\cap{l'}}.
\end{equation}
Once \eqref{x'-x'''} is established, the proof is completed by combining \eqref{x'-x'''} and \eqref{eq:x'''} to get
\begin{equation}
\left\|x'-x\right\|\leq\frac{(m_{n}+4)\Y^{2}}{\psi s_{n}}\times\delta\leq\eta\delta,
\end{equation}
where the final inequality follows from \eqref{eq:delta0}, $\delta\in(0,\delta_0')$ and \eqref{eq:n}.

Thus, it only remains to verify \eqref{x'-x'''}. We distinguish two cases; namely the case $t=0$ and the case $t\geq{1}$.

If $t=0$ then the conditions of Lemma~\ref{lemma:c1} are satisfied for $\lambda$, $\psi$, $x$, $\delta$, $t$, $n$, $f=e'$, $l=l'''$, $\myr=1+m_{n}$, $(h_{1},\ldots,h_{\myr})=(j_{1},\ldots,j_{1+m_{n}})$ and $y=x'''\in{l'''}$. Therefore, by Lemma~\ref{lemma:c1}, there exists a line $l'$ of level~$n$, class~$2+m_{n}$ category~$(j_{1},\ldots,j_{1+m_{n}},1)$ and point $x'\in{l'}$ such that
\begin{align}
&\left\|x'-x'''\right\|\leq\frac{\Y}{\psi s_{n}}\times\delta\leq\frac{(m_{n}+2)\Y^{2}}{\psi s_{n}}\times\delta\mbox{ and,} \label{eq:x'1} \\
&x'+[-1,1]\tau e'\subseteq{\uds_{\lambda+\psi}\cap{l'}}\mbox{ whenever }
0\leq
\tau\leq(\Y-\frac{\Y}{\psi s_{n}})\delta-\left\|x'''-x\right\|. \label{eq:itvl}
\end{align}
 From \eqref{eq:delta0}, $\delta\in(0,\delta_{0}')$ and \eqref{eq:n} we can deduce that
 \begin{equation}\label{eq:etabounds}
 \frac{\Y^{2}}{\psi s_{n}}\leq\frac{\eta}{3}\mbox{ and }\frac{\Y^{2}(m_{n}+4)}{\psi s_{n}}\leq\eta.
 \end{equation}
Therefore, using $\eta<1/4$, \eqref{eq:etabounds} and \eqref{eq:x'''} 
we get 
 \begin{equation*}
 (\Y-\frac{\Y}{\psi s_{n}})\delta-\left\|x'''-x\right\|\geq\left(\Y-\frac{\eta}{3}-{\eta}\right)\delta\geq\frac\delta2. 
 \end{equation*}
 Hence, by \eqref{eq:itvl} we have $x'+[-1,1]\frac\delta2e'\subseteq{\uds_{\lambda+\psi}\cap{l'}}$, and we obtain \eqref{x'-x'''}.

Now assume that we are in the remaining case, $t\geq{1}$. Set $i_{1+m_n}=t+1$, so that $j_{1+m_n}=1<i_{1+m_n}\leq s_n$. Observe that the line $l'''\in\mathcal{L}_{(n,1+m_{n})}^{(j_{1},\ldots,j_{1+m_{n}})}$, the integer $i_{1+m_n}>j_{1+m_n}$ and the point $x'''\in{l'''}$ satisfy the conditions of Lemma~\ref{approx}. Therefore, by Lemma~\ref{approx}, there exists an integer sequence $s_k\geq i_1\geq\ldots\geq i_{1+m_n}\geq1$ together with a line
$l''$ of level~$n$, class~$1+m_n$, category~$(i_{1},\ldots,i_{1+m_n})$
such that $l''$ is parallel to $l'''$ and there exists a point $x''\in{l''}$ with
\begin{equation}\label{x''-x'''}
\left\|x''-x'''\right\|\leq\frac{(1+m_n)\times{\Y^{t+1}w_{n}}}{s_{n}}.
\end{equation}
Set $i_{2+m_n}=t+1$, so that $i_{2+m_n}=i_{1+m_n}$. Note that the conditions of Lemma~\ref{lemma:c1} are satisfied for $\lambda$, $\psi$, $x$, $\delta$, $t$, $n$, $f=e'$, $l=l''$, $\myr=1+m_{n}$, $(h_{1},\ldots,h_{\myr})=(i_{1},\ldots,i_{1+m_{n}})$ and $y=x''\in{l''}$. Hence, by Lemma~\ref{lemma:c1}, there exists a line segment $l'$ of
level~$n$, class~$2+m_n$, category~$(i_{1},\ldots,i_{1+m_n},t+1)$ and a point $x'\in{l'}$ with
\begin{align}
&\left\|x'-x''\right\|\leq\frac{\Y}{\psi s_{n}}\times\delta\mbox{ and,} \label{x'-x''} \\
&x'+[-1,1]\tau e'\subseteq{\uds_{\lambda+\psi}\cap{l'}}\mbox{ whenever }
0\leq
\tau\leq\left(\Y-\frac{\Y}{\psi s_{n}}\right)-\left\|x''-x\right\|. \label{eq:itvl2}
\end{align}
We observe that 
\begin{equation*}
\left\|x'-x'''\right\|\leq\frac{(m_{n}+2)\Y^{t+1}w_{n}}{s_{n}}\leq\frac{(m_{n}+2)\Y^{2}}{\psi s_{n}}\times\delta,
\end{equation*}
using \eqref{x'-x''}, \eqref{x''-x'''} and \eqref{eq:n}. Moreover, combining \eqref{x''-x'''} with \eqref{eq:x'''} yields
\begin{equation*}
\left\|x''-x\right\|\leq\frac{(3+m_{n})\Y^{t+1}w_{n}}{s_{n}}\leq\frac{(m_{n}+3)\Y^{2}}{\psi s_{n}}\times\delta.
\end{equation*}
Therefore, by \eqref{eq:etabounds} and $\eta<1/4$,
\begin{equation*}
\left(\Y-\frac{\Y}{\psi s_{n}}\right)\delta-\left\|x''-x\right\|\geq\left(\Y-\frac{(m_{n}+4)\Y^{2}}{\psi s_{n}}\right)\delta\geq(\Y-\eta)\delta\geq\frac\delta2.
\end{equation*}
We conclude, using \eqref{eq:itvl2} that 
$x'+[-1,1]\frac\delta2 e'\subseteq \uds_{\lambda+\psi}\cap l'$. We have now verified \eqref{x'-x'''}.
\end{proof}
\begin{lemma}\label{lemma:c3}
Let $\lambda\in[0,1)$, $\psi\in(0,1-\lambda)$ and $\eta\in(0,1)$. Then there exists a number
\begin{equation*}
\delta_{1}=\delta_{1}(\lambda,\psi,\eta)>0
\end{equation*}
such that whenever $x\in{\uds_{\lambda}}$, $\delta\in(0,\delta_{1})$ and $v_{1},v_{2},v_{3}$ are in the closed unit ball in $\mathbb{R}^{d}$, there exist ${v_{1}}',{v_{2}}',{v_{3}}'\in\mathbb{R}^{d}$ such that 
\begin{align}
&\left\|{v_{i}}'-v_{i}\right\|\leq\eta\mbox{ and}  \label{c3a} \\
&[x+\delta{v_{1}}',x+\delta{v_{3}}']\cup[x+\delta{v_{3}}',x+\delta{v_{2}}']\subseteq{\uds_{\lambda+\psi}}. \label{c3b}
\end{align}
Moreover, $\delta_1$ can be chosen to be independent on $\Y\in(1,2)$.
\end{lemma}
\begin{proof}
Fix positive numbers $\mya,\myb,\myc$ such that 
\begin{equation}\label{eq:myabc}
\mya+2\myb+3\myc<\frac12.
\end{equation}

Using the notation of Lemma~\ref{lemma:crucial}, choose $0<\delta_{1}\le\delta_{0}\left(\lambda,{\psi},\mya\eta\right)$  
such that
\[
\frac{2}{\psi s_{k}}\leq\myb{\eta}\mbox{ whenever }\psi{w_{k}}<2\delta_{1} 
\]
implying that 
\begin{equation}
\frac{\Y}{\psi s_{k}}\leq\myb{\eta}\mbox{ whenever }\psi{w_{k}}<\Y\delta_{1} \label{eq:del1b}
\end{equation}
as $\Y\in(1,2)$.

Fix $x\in{\uds_{\lambda}}$, $\delta\in(0,\delta_{1})$ and $v_{1},v_{2},v_{3}$ in the closed unit ball in $\mathbb{R}^{d}$. We may assume that 
\begin{equation}\label{eq:vse1}
0<\left\|v_{i}\right\|\leq \myc\mbox{ for each }i=1,2,3
\end{equation}
and $v_1,v_2,v_3$ are distinct vectors.

Set $e_1=v_{1}/\|v_{1}\|$. Since $\delta<\delta_{0}(\lambda,{\psi},\mya\eta)$, Lemma~\ref{lemma:crucial} asserts that there exists  $e_{1}'\in{S^{d-1}}$,
integers $n,t$  
and $x'\in l'\in\mathcal L_{(n,\myr)}^{(h_1,\dots,h_\myr)}$ 
such that \eqref{eq:tnd} and \eqref{eq:ch:str} are satisfied, together with
\begin{equation}\label{eq:x'e1'}
\left\|x'-x\right\|\leq\mya\eta\delta\mbox{, }
\left\|{e_{1}}'-e_{1}\right\|\leq\mya{\eta}\mbox{ and }
x'+[-1,1]\delta e_1'\subseteq \uds_{\lambda+\psi}\cap l'.
\end{equation}
Denote $l_1:=l'$ and 
set 
\begin{equation}\label{eq:x1e2}
x_{1}=x'+\delta\left\|v_{1}\right\|{e_{1}}'\mbox{ and }
e_{3}=(v_{3}-v_{1})/\left\|v_{3}-v_{1}\right\|.
\end{equation}
Let $e_3'\in\mathcal E_n$ be such that $\|e_3'-e_3\|\le1/{s_n}$. We note that  
\eqref{eq:del1b} implies $1/s_n\le\Y/(\psi s_n)\le\myb\eta$, as by \eqref{eq:tnd}
we have $\psi w_n\le  \psi\Y^t w_n<\Y\delta<\Y\delta_1$.
We now apply Lemma~\ref{lemma:c1} to point $x\in \uds_\lambda$, integers 
$n,t$ found above, $\delta$ satisfying \eqref{eq:tnd}, $f:=e_3'$,
$y:=x_1\in[x',x'+\delta e_1']\subseteq l_1\in\mathcal L_{(n,\myr)}^{(h_1,\dots,h_\myr)}$. Let the point $y'\in l_1$ and the line
$l_1'\in\mathcal L_{(n,1+\myr)}^{(h_1,\dots,h_\myr,h_\myr)}$ be given by the conclusion of
Lemma~\ref{lemma:c1}. 

We now define $x_1'=y'$ and 
note that \eqref{eq:y-y'} and \eqref{eq:del1b} imply
\[\|x_1'-x_1\|=
\|y'-x_1\|
\le\myb{\eta}\delta,\] 
so that using \eqref{eq:vse1}
we get $\|x_1'-x'\|\le(\myb\eta+\myc)\delta$. 

We claim that
the straight line segment $[x_1'-\frac12\delta e_3',x_1'+\frac12\delta e_3']$
is inside $\uds_{\lambda+\psi}$.
Indeed, we verify that $\tau=\delta/2$ satisfies \eqref{eq:tau}.
Using 
\begin{equation}\label{eq:x1-x}
\|x_1-x\|
\le\|x_1-x'\|+\|x'-x\|
\le
(\myc+\mya\eta)\delta
\end{equation}
and $\Y>1$, together with \eqref{eq:myabc} and $0<\eta<1$, we get
\[
(\Y-\frac\Y{\psi s_n})\delta-\|x_1-x\|\ge
(1-\myb\eta)\delta-(\myc+\mya\eta)\delta
>\frac\delta2.
\]

Let
\[x_3=x_1'+\delta\|v_3-v_1\| e_3'.\]
Denote $l_3=l_1'$ and set $e_2=(v_2-v_3)/\|v_2-v_3\|$.
Find $e_2'\in\mathcal E_n$ with $\|e_2'-e_2\|\le\frac1{s_n}$ and apply
Lemma~\ref{lemma:c1} to point $x\in \uds_\lambda$, $n,t$ and $\delta$ satisfying \eqref{eq:tnd} and found earlier, $f:=e_2'$,
$y:=x_3\in[x_1',x_1'+\delta e_3']\subseteq l_3\in\mathcal L_{(n,1+\myr)}^{(h_1,\dots,h_\myr,h_\myr)}$. 
We note that the condition \eqref{eq:ch}
of Lemma~\ref{lemma:c1}  is satisfied for $\myr+1$ instead of $\myr$ because of
\eqref{eq:ch:str}.
Let the point $y'\in l_3$ and the line
$l_2\in\mathcal L_{(n,2+\myr)}^{(h_1,\dots,h_\myr,h_\myr,h_\myr)}$ be given by the conclusion of
Lemma~\ref{lemma:c1}. 

Let $x_3'=y'$.
We now verify that $[x_3'-\frac12\delta e_2',x_3'+\frac12\delta e_2']\subseteq \uds_{\lambda+\psi}$. 
We again show that
$\tau=\delta/2$ satisfies \eqref{eq:tau}.
Indeed using \eqref{eq:x1-x} we get
\begin{multline*}
\|x_3-x\|\le\|x_3-x_1'\|+\|x_1'-x_1\|+\|x_1-x\|\\
\le
2\myc\delta+\myb\eta\delta+(\myc+\mya\eta)\delta
=
(\mya\eta+\myb\eta+3\myc)\delta.
\end{multline*}
Hence, using \eqref{eq:myabc} and $0<\eta<1$, we conclude
\[
(\Y-\frac\Y{\psi s_n})\delta-\|x_3-x\|\ge
(1-\myb\eta)\delta-(\mya\eta+\myb\eta+3\myc)\delta
>\frac\delta2.
\]

Finally,
define 
\[x_2'=x_3'+\|v_2-v_3\|\delta e_2'.\]
We are now left to see that
$v_i'$, $i=1,2,3$ defined according to
\begin{equation}\label{eq:defv'}
x+\delta v_i'=x_i'
\iff
v_i'=(x_i'-x)/\delta
\end{equation}
satisfy the conclusions of Lemma~\ref{lemma:c3}.

Indeed, let us verify $[x_1',x_3']\cup[x_3',x_2']\subseteq \uds_{\lambda+\psi}$.
First we see that $x_3'\in l_3$ and by \eqref{eq:myabc}
\[
\|x_3'-x_1'\|\le\|x_3'-x_3\|+\|x_3-x_1'\|
\le
\myb\eta\delta+2\myc\delta
<\delta/2,
\]
hence
$[x_1',x_3']\subseteq[x_1'-\frac12\delta e_3',x_1'+\frac12\delta e_3']\subseteq \uds_{\lambda+\psi}$.
For the second straight line segment we see that 
$\|x_2'-x_3'\|\leq2\myc\delta$ and $x_2'\in l_2$ so that 
$[x_3',x_2']\subseteq [x_3'-\frac12\delta e_2',x_3'+\frac12\delta e_2']\subseteq \uds_{\lambda+\psi}$.

By \eqref{eq:defv'} we see that \eqref{c3a} is equivalent to 
\[
\|(x_i'-x_i)-\delta v_i\|\le\eta\delta
\mbox{ for all }
i=1,2,3.
\]
We note first that using 
\eqref{eq:myabc}
\begin{multline*}
\|(x_1'-x)-\delta v_1\|
\le 
\|(x_1-x')-\delta v_1\|+\|x_1'-x_1\|+\|x'-x\|\\
\le 
\myc\delta\|e_1'-e_1\|+(\mya+\myb)\eta\delta
\le(\mya+\myb+\mya\myc)\eta\delta
<\eta\delta.
\end{multline*}
Next, 
\begin{multline*}
\|(x_3'-x)-\delta v_3\|
\le
\|x_3'-x_3\|+\|(x_3-x_1')-\delta(v_3-v_1)\|+\|(x_1'-x)-\delta v_1\|\\
\le
\myb\eta\delta+\delta\|v_3-v_1\|\times\|e_3'-e_3\|+(\mya+\myb+\mya\myc)\eta\delta
\\
\le (\mya+2\myb+\mya\myc+2\myb\myc)\eta\delta<\eta\delta
\end{multline*}
using $\|v_3-v_1\|\le2\myc$ and $\|e_3'-e_3\|\le\myb\eta$.
Finally, using the definition of $x_2'$, we get
\begin{multline*}
\|(x_2'-x)-\delta v_2\|
=
\|(x_3'-x)+\delta\|v_2-v_3\|e_2'-\delta v_2\|\\
\le
\|(x_3'-x)+\delta(v_2-v_3)-\delta v_2\|
+\delta\|v_2-v_3\|\times\|e_2'-e_2\|\\
=\|(x_3'-x)-\delta v_3\|+\delta\|v_2-v_3\|\times\|e_2'-e_2\|
\le (\mya+2\myb+\mya\myc+4\myb\myc)\eta\delta
<\eta\delta
\end{multline*}
as $\mya+2\myb+\mya\myc+4\myb\myc<2(\mya+2\myb+3\myc)<1$.
\end{proof}

We are now ready to prove our main result.
\begin{theorem}\label{thm:main}
For every $d\ge1$,
there exists a compact subset $S\subseteq\mathbb{R}^{d}$ of upper Min\-kow\-ski dimension one with the universal differentiability property.  Moreover if $g:\mathbb{R}^{d}\to\mathbb{R}$ is Lipschitz, the set of points $x\in{S}$ such that $g$ is Fr\'{e}chet differentiable at $x$ is a dense subset of $S$.
\end{theorem}
\begin{proof}
From Lemma~\ref{lemma:c3}, we have that the family of closed sets $(\uds_{\lambda})$, ${\lambda\in[0,1]}$ satisfy the conditions of Theorem~\ref{thm:udscrit}. Therefore, by Theorem~\ref{thm:udscrit}, the set
\begin{equation*} 
S=\overline{\bigcup_{q<1}\uds_{q}}
\end{equation*}
is a universal differentiability set. Moreover, Theorem~\ref{thm:udscrit} asserts that whenever $g$ is a Lipschitz function on $\mathbb{R}^{d}$, the set of points $x\in{S}$ such that $g$ is differentiable at $x$ is a dense subset of $S$. All that remains is to show that $S$ has upper Minkowski dimension one. This follows from the observation that $\uds_{1/2}\subseteq{S}\subseteq{\uds_{1}}$, together with Lemma~\ref{lemma:dimension}.
\end{proof}
 
\bibliographystyle{plain}
\bibliography{DymondMaleva-biblio}

\end{document}